\documentclass[12pt,reqno]{amsart}
\usepackage{amsmath, amsthm, amssymb}
\usepackage{enumerate}
\usepackage{mathrsfs}
\usepackage{color}
\usepackage[english]{babel}
\usepackage{graphicx}
\usepackage{hyperref}
\usepackage{float}

\usepackage{multirow}
\usepackage{psfrag}
\usepackage[ansinew]{inputenc}

\def\({\left(}
\def\){\right)}

\def\RR{{\mathbb{R}}}
\def\CC{{\mathbb{C}}}

\newcommand{\be}{\begin{equation}}
\newcommand{\ee}{\end{equation}}
\newcommand{\bea}{\begin{eqnarray}}
\newcommand{\eea}{\end{eqnarray}}
\newcommand{\beann}{\begin{eqnarray*}}
\newcommand{\eeann}{\end{eqnarray*}}

 \newtheorem{theorem}{Theorem}

\newtheorem{remark}{Remark}

\newtheorem{lemma}{Lemma}[section]

\newenvironment{proof1}{
    \noindent {\em Proof }}{\hfill$\Box$}

\newcommand{\bn}{\begin{eqnarray*}}       
\newcommand{\en}{\end{eqnarray*}}
\newcommand{\beq}{\begin{equation}}
\newcommand{\eeq}{\end{equation}}

\newcommand{\R}{\mathbb{R}}

\numberwithin{equation}{section}

\begin{document}

\title[Black soliton in a regularized NLS equation]{\bf Dynamics of the black soliton \\ in a regularized nonlinear Schr\"{o}dinger equation}

\author{Dmitry E. Pelinovsky}
\address[D.E. Pelinovsky]{Department of Mathematics and Statistics, McMaster University,	Hamilton, Ontario, Canada, L8S 4K1}
\email{dmpeli@math.mcmaster.ca}

\author{Michael Plum}
\address[M. Plum]{Institut f\"{u}r Analysis, Karlsruher Institut f\"{u}r Technologie, Karlsruhe, Germany, 76131}
\email{michael.plum@kit.edu}

\begin{abstract}
	We consider a family of regularized defocusing nonlinear Schr\"{o}dinger (NLS) equations proposed in the context of the cubic NLS equation with a bounded dispersion relation. The time evolution is well-posed if the black soliton is perturbed by a small perturbation in the Sobolev space $H^s(\R)$ with $s > \frac{1}{2}$. We prove that the black soliton is spectrally stable (unstable) if the regularization parameter is below (above) some explicitly specified threshold. We illustrate the stable and unstable dynamics of the perturbed black solitons by using the numerical finite-difference method. The question of orbital stability of the black soliton is left open due to the mismatch of the function spaces for the energy and  momentum conservation.
\end{abstract}

\date{\today}
\maketitle


\section{Introduction}
\label{sec-1}

Dark solitons are the depression waves propagating steadily along the continuous wave background. The name is drawn from the realms of nonlinear optics, where the continuous wave background supports the light of constant intensity and the dark solitons reduce the light intensity during transmission \cite{Kivshar,China}. The most extreme case in the family of dark solitons is the black soliton, for which the light intensity drops to zero and the wave is spatially localized with zero speed. Such standing waves are common in many models of nonlinear optics and Bose--Einstein condensation in one, two,  and three spatial dimensions \cite{Kev-book}.

The canonical model for the dark and black solitons is the defocusing nonlinear Schr\"{o}diger (NLS) equation 
\begin{equation}
\label{nls-canon}
i \psi_t + \psi_{xx} - 2 |\psi|^2 \psi = 0,
\end{equation}
where $\psi(t,x) : \RR \times \RR \to \CC$. The exact traveling wave solutions for the dark solitons are given by 
\begin{equation}
\label{dark-cubic}
\psi(t,x) = \left[ \gamma \tanh(\gamma(x-2ct)) + i c \right] e^{-2it}, \quad \gamma := \sqrt{1-c^2},
\end{equation}
where $c \in (-1,1)$ is a free parameter for the half wave speed. 
For $c = 0$, the dark soliton $\psi(t,x) = \tanh(x) e^{-2it}$ is referred as the black soliton. 

Compared to the canonical NLS equation (\ref{nls-canon}), we address the following family of regularlized defocusing NLS equations 
\begin{equation}
\label{nls-ext}
i (1 - \mu^2 \partial_x^2) \psi_t + \psi_{xx} - 2 |\psi|^2 \psi = 0,
\end{equation}
where $\psi(t,x) : \RR \times \RR \to \CC$ and $\mu > 0$ is a small parameter. 
The model was proposed in the context of nonlinear optics for ultra-short pulses \cite{ColinLannes,Lannes-review} (see \cite[Section 4.1.4, eq. (72)]{Lannes}) as an example of the NLS equation with a bounded dispersion relation. The linear part of the model has the dispersion relation 
$$
\omega(k) = \frac{k^2}{1 + \mu^2 k^2}, \quad k \in \R,
$$
obtained for the Fourier modes $u(t,x) \sim e^{i k x - i \omega(k) t}$. 
The additional term $- i \mu^2 \psi_{txx}$ in (\ref{nls-ext}) compared to the canonical model (\ref{nls-canon}) leads to the bounded dispersion relation 
in the interval $[0,\mu^{-2}]$ for the frequencies $\omega$. This kind of regularization of the dispersion relation is popular in models of fluid dynamics where 
the canonical Korteweg--de Vries equation is replaced by a regularized Benjamin--Bona--Mahony equation \cite{BBM}.

We are mainly interested in the stability of the black soliton which is the  
standing wave solution of the defocusing NLS equation (\ref{nls-ext}) with nonzero boundary conditions. If the black soliton is stable in the time 
evolution, as is known for the canonical NLS equation (\ref{nls-canon}) 
\cite{BGSS,Chiron,GPII,PZ,GS,GalloMenza}, then it plays an important role in 
nonlinear optics as a carrier of information inside the spatially modulated periodic waves \cite{Maiden,Sprenger} (see review in \cite{Hoefer}).

In order to study the black soliton of the model (\ref{nls-ext}), we normalize the boundary conditions to unity without  loss of generality and consider solutions satisfying
$$
|\psi(t,x)| \to 1 \quad \mbox{\rm as} \;\; |x| \to \infty.
$$
With the transformation 
\begin{equation}
\label{transform}
\psi(t,x) = e^{-2it} u(t,\xi), \quad \xi = \frac{x}{\sqrt{1-2\mu^2}},
\end{equation}
the NLS equation (\ref{nls-ext}) can be rewritten in the equivalent form as 
\begin{equation}
\label{nls}
i (1 - \epsilon^2 \partial_{\xi}^2) u_t + u_{\xi \xi} + 2 (1- |u|^2) u = 0,
\end{equation}
where 
$$
\epsilon := \frac{\mu}{\sqrt{1-2\mu^2}}.
$$ 
The mapping $\mu \to \epsilon$ is monotonically increasing for $\mu \in (0,\frac{1}{\sqrt{2}})$ with $\epsilon \to \infty$ as $\mu \to \frac{1}{\sqrt{2}}$.

The black soliton is the steady-state (time-independent) solution
of the transformed NLS equation (\ref{nls}). It is available 
in the explicit form $u(t,\xi) = \varphi(\xi) := \tanh(\xi)$, which coincides with the black soliton of the defocusing NLS equation (\ref{nls-canon}) 
given by (\ref{dark-cubic}) for $c = 0$.

Local well-posedness of the regularized model (\ref{nls}) can be studied 
in the space of bounded smooth functions with nonzero boundary conditions at infinity. One general method is to consider a decomposition 
$u(t,\xi) =  \varphi(\xi) + v(t,\xi)$, where the perturbation 
$v(t,\cdot)$ is a continuous function of $t$ 
in Sobolev spaces with respect to the spatial coordinate \cite{Gallo,Gerard,Zhid}.  This is our first result presented in the following theorem. 

\begin{theorem}
	\label{thr-local-well}
	For every $v_0 \in H^s(\mathbb{R})$ with $s > \frac{1}{2}$, there exists the maximal existence time $\tau_0 \in (0,\infty]$ and a unique solution of the NLS equation (\ref{nls}) in the form $u = \varphi + v$, where 
	$\varphi(\xi) = \tanh(\xi)$ and $v \in C^1([0,\tau_0),H^s(\mathbb{R}))$ such that $v(0,\cdot) = v_0$. Moreover, $\tau_0 \in (0,\infty]$ and $v \in C^1([0,\tau_0),H^s(\mathbb{R}))$ depend continuously on $v_0 \in H^s(\mathbb{R})$.
\end{theorem}

Replacing the solution $u$ of Theorem \ref{thr-local-well} with $\varphi +v$, where $\varphi(\xi) = \tanh(\xi)$ is the black soliton and $v := U + i V$ is a small perturbation allows us to reformulate the stability of the black soliton at the level of linearized approximation. Separation of the real part $U$ and the imaginary part $V$ gives the following linearized equations:
\begin{equation}
\label{nls-lin}
(1 - \epsilon^2 \partial_{\xi}^2) U_t = L_- V, \qquad 
(1 - \epsilon^2 \partial_{\xi}^2) V_t = -L_+ U,
\end{equation}
where the linear operators $L_{\pm}$ in $L^2(\R)$ with ${\rm Dom}(L_{\pm}) \subset L^2(\R)$ are defined by the differential expressions 
\begin{align*} 
L_+ &= -\partial_{\xi}^2 + 6 \varphi^2 - 2, \\
\label{L-minus}
L_- &= -\partial_{\xi}^2 + 2 \varphi^2 - 2.
\end{align*}
Separation of variables in system (\ref{nls-lin}) gives the spectral stability problem in the form 
\begin{equation}
\label{lin-stab}
\left[ \begin{array}{cc} 0 & L_- \\ -L_+ & 0 \end{array} \right] 
\left[ \begin{array}{c} U \\ V \end{array} \right] = \lambda (1 - \epsilon^2 \partial_{\xi}^2)
\left[ \begin{array}{c} U \\ V \end{array} \right] \qquad \Leftrightarrow \qquad 
\begin{array}{c} L_- V = \lambda (1 - \epsilon^2 \partial_{\xi}^2)  U, \\ -L_+ U = \lambda (1 - \epsilon^2 \partial_{\xi}^2) V.\end{array}
\end{equation}
It is natural to consider the spectral problem (\ref{lin-stab}) 
for fixed $\epsilon \neq 0$ in $H^1_{\epsilon}(\R) \times H^1_{\epsilon}(\R)$, where  $H^1_{\epsilon}(\R)$  is the Hilbert space
equipped with the associated inner product 
\begin{equation}
\label{inner-product}
(f,g)_{\epsilon} := \int_{\R} \left[ \bar{f} g + \epsilon^2 \bar{f}' g' \right] d \xi.
\end{equation}
and the induced norm $\| \cdot \|_{\epsilon} := \sqrt{ (\cdot,\cdot)_{\epsilon}}$. The Hilbert space $H^1_{\epsilon}(\R)$ for fixed $\epsilon \neq 0$ is equivalent to $H^1(\RR)$, which is 
the form domain of $L_+$ in $L^2(\R)$ and is a subset of the form domain of $L_-$ in $L^2(\RR)$. The inner product and the norm in $L^2(\R)$ correspond to $\epsilon = 0$ and we use notations $(\cdot,\cdot)$ and $\| \cdot \|$ instead of $(\cdot,\cdot)_0$ and $\| \cdot \|_0$.

The essential spectrum of the spectral stability problem (\ref{lin-stab}) can be found in the limit $|\xi| \to \infty$ since $|\varphi(\xi)| \to 1$ exponentially fast. By using the Fourier transform, we find that the essential spectrum is located at 
\begin{equation}
\sigma_c  = \left\{ i k \frac{\sqrt{4+k^2}}{1 + \epsilon^2 k^2}, \quad k \in \mathbb{R} \right\} = i [-\epsilon^{-2},\epsilon^{-2}].
\end{equation}
Hence, the essential spectrum is neutrally stable and the stability or instability of the black soliton depends on isolated eigenvalues $\lambda$ outside $\sigma_c$.

\begin{remark}
	We say that the black soliton $\varphi$ is spectrally unstable if the spectral problem (\ref{lin-stab}) admits an isolated eigenvalue $\lambda_0 \in \CC$ with the corresponding eigenfunction $(U,V) \in H^1(\R) \times H^1(\R)$ such that ${\rm Re}(\lambda_0) > 0$. The black soliton is spectrally stable if no such eigenvalue exists. 
\end{remark}

\begin{remark}
	Isolated eigenvalues may exist on $i \R\backslash \sigma_c$ but such eigenvalues do not contribute to spectral instability of the black soliton.
\end{remark}

Our second result is the following stability theorem, which is the main result of this paper.

\begin{theorem}
	\label{thm-spectral}
	Let $\epsilon_0 = (5/8)^{1/4}$. The black soliton is spectrally stable 
	for $\epsilon \in (0,\epsilon_0]$ and spectrally unstable for $\epsilon \in (\epsilon_0,\infty)$.
\end{theorem}

\begin{remark}
	The stability threshold $\epsilon_0 = (5/8)^{1/4}$ is approximately 
	$\epsilon_0 \approx 0.89$. 	In view of the transformation (\ref{transform}), the stability threshold in the original NLS model (\ref{nls-ext}) is given by 
	$$
	\mu_0 = \sqrt{\frac{\sqrt{5}}{\sqrt{8} + 2 \sqrt{5}}} \approx 0.5534. 
	$$
	The black soliton is orbitally stable for $\mu \in (0,\mu_0)$ 
	and orbitally unstable for $\mu \in (\mu_0,\mu_1)$, 
	where $\mu_1 = \frac{1}{\sqrt{2}} \approx 0.7071$.
\end{remark}

It is tempting to extend the spectral stability of Theorem \ref{thm-spectral} 
to the orbital stability of the black soliton similar to \cite{BGSS,GPII,GS} for the canonical NLS equation (\ref{nls-canon}). The proof of orbital stability of the black soliton was also developed for other NLS models such as 
the quintic NLS equation \cite{Alejo}, the coupled NLS systems \cite{CPP18}, and the NLS equation with intensity-dependent dispersion \cite{Plum}. Orbital stability is usually proven with the use of conserved quantities. Our third result specifies the conserved quantities of the regularized NLS equation (\ref{nls}). 

\begin{theorem}
	\label{thm-conservation} Let $u = \varphi + v$ with $v \in C^1([0,\tau_0),H^s(\mathbb{R}))$ be the local solution in Theorem \ref{thr-local-well} with $s > \frac{1}{2}$. Then, energy
	\begin{equation}
	\label{energy}
	E(u) = \int_{\mathbb{R}} \left[ |u_{\xi}|^2 + (1-|u|^2)^2 \right] d\xi.
	\end{equation} 
	and momentum 
	\begin{equation}
	\label{momentum}
	P(u) = i \int_{\mathbb{R}} \left[ (\bar{u} u_{\xi} - \bar{u}_{\xi} u) + 
	\epsilon^2 (\bar{u}_{\xi} u_{\xi \xi} - \bar{u}_{\xi \xi} u_{\xi}) \right] d\xi.
	\end{equation}
	are well defined for $s \geq 1$ and $s \geq 2$ respectively, and their values are independent of $t \in [0,\tau_0)$.
\end{theorem}

\begin{remark}
	Besides the energy and momentum conservation, the NLS equation (\ref{nls}) admits also mass conservation, 
	\begin{equation}
	\label{mass}
	M(u) = \int_{\mathbb{R}} \left[  \epsilon^2 |u_{\xi}|^2 + |u|^2  - 1 \right] d\xi,
	\end{equation}
if the solution $u = \varphi + v$ with $v \in C^1([0,\tau_0),H^s(\mathbb{R}))$ and $s \geq 1$ satisfies $v(t,\cdot) \in L^1(\mathbb{R})$ for $t \in [0,\tau_0)$. The mass conservation plays no role in the proof of orbital stability of the black soliton in the canonical NLS equation (\ref{nls-canon}) \cite{BGSS,GS}.
\end{remark}

\begin{remark}
	The proof of orbital stability of the black soliton is an open problem for the regularized NLS equation (\ref{nls}) because of the mismatch between the energy and momentum spaces. The energy arguments only provide control of the perturbation in the weighted $H^1(\mathbb{R})$ spaces with the exponential weight \cite{GS}, where the weight is needed due to the lack of coercivity of the quadratic form associated with the linearized operator $L_-$, see also \cite{Alejo,CPP18,Plum}. However, the momentum conservation is not defined in the weighted $H^1(\R)$ space and the energy arguments do not control perturbations in the weighted $H^2(\R)$ space. 
\end{remark}

The remainder of this paper is organized as follows. Section \ref{sec-2}  presents the proof of Theorems \ref{thr-local-well} and \ref{thm-conservation}. Section \ref{sec-3} is devoted to the proof of Theorem \ref{thm-spectral}. 
The numerical illustrations of the stable and unstable dynamics of the perturbed black soliton are contained in Section \ref{sec-4}. Section \ref{sec-5} concludes the paper with a summary and the discussion of further directions.

\section{Local well-posedness and conserved quantities}
\label{sec-2}

Let us write $u = \varphi + v$, where $\varphi(\xi) = \tanh(\xi)$, 
and reduce the NLS equation (\ref{nls}) to the evolutionary form:
\begin{equation}
\label{nls-evol}
v_t = i (1-\epsilon^2 \partial_{\xi}^2)^{-1} \left[ v_{\xi \xi} 
+ 2(1-2\varphi^2) v - 2 \varphi^2 \bar{v} - 2 \varphi (v^2 + 2 |v|^2) - 2 |v|^2 v \right],
\end{equation}
where we have used $\varphi'' + 2 (1-\varphi^2) \varphi = 0$ satisfied by 
$\varphi(\xi) = \tanh(\xi)$. The proof of Theorem \ref{thr-local-well} follows from the contraction mapping 
principle according to the following arguments. 

\vspace{0.25cm}

\begin{proof1}{\em of Theorem \ref{thr-local-well}.}
We recall that $H^s(\R)$ forms a Banach algebra with respect to pointwise multiplication if $s > \frac{1}{2}$. Hence, 
$$
F(v) := v_{\xi \xi} 
+ 2(1-2\varphi^2) v - 2 \varphi^2 \bar{v} - 2 \varphi (v^2 + 2 |v|^2) - 2 |v|^2 v
$$
is a bounded operator from $H^s(\R)$ to $H^{s-2}(\R)$, whereas 
$(1 - \epsilon^2 \partial_{\xi}^2)^{-1}$ is a bounded operator from $H^{s-2}(\R)$ to $H^s(\R)$ for every $\epsilon > 0$. It follows from the integral formulation 
of the evolution equation (\ref{nls-evol}) and the contraction mapping principle that there exists a local solution $v \in C^0([0,t_0],H^s(\R))$ 
for sufficiently small $t_0 > 0$. Moreover, since the operator 
$i (1 - \epsilon^2 \partial_{\xi}^2)^{-1} F : H^s(\R) \to H^s(\R)$ is bounded in a ball of $H^s(\R)$ of a finite radius, then the solution $v$ belongs to $C^1([0,t_0],H^s(\R))$. Continuing smoothly the local solution to the maximal time $\tau_0 > 0$ (which could be finite or infinite) yields the solution $v \in C^1([0,\tau_0),H^s(\R))$ in Theorem \ref{thr-local-well}. 

Continuous dependence of $v \in C^1([0,t_0],H^s(\R))$ on $v_0 \in H^s(\R)$ is obtained from the contraction principle by the standard Gronwall's estimates. Iterating the estimates to the maximal existence time, we obtain 
continuous dependence of $\tau_0 \in (0,\infty]$ and  $v \in C^1([0,\tau_0),H^s(\R))$ on $v_0 \in H^s(\R)$.
\end{proof1}

\vspace{0.25cm}

By using the transformation $u = \varphi + v$, we can write from (\ref{energy}), (\ref{momentum}), and (\ref{mass}):
\begin{align*}
E(\varphi + v) &= E(\varphi) + \hat{E}(v), \\
P(\varphi + v) &= P(\varphi) + \hat{P}(v), \\
M(\varphi + v) &= M(\varphi) + \hat{M}(v),
\end{align*}
with 
\begin{align*}
\hat{E}(v) &= \int_{\mathbb{R}} [|v_{\xi}|^2 - 2(1-2 \varphi^2) |v|^2 
+ \varphi^2 (v^2 + \bar{v}^2) + 2 \varphi |v|^2 (v + \bar{v}) + |v|^4] d \xi, \\
\hat{P}(v) &= i \int_{\mathbb{R}} [2 \varphi' (\bar{v} - v) + 
(\bar{v} v_{\xi} - \bar{v}_{\xi} v) + 2 \epsilon^2 
\varphi'' (\bar{v}_{\xi} - v_{\xi}) + \epsilon^2 (\bar{v}_{\xi} v_{\xi \xi} - \bar{v}_{\xi \xi} v_{\xi})] d \xi, \\
\hat{M}(v) &= \int_{\mathbb{R}} [(\varphi -\epsilon^2 \varphi'') (v + \bar{v})
+ \epsilon^2 |v_{\xi}|^2 + |v|^2] d \xi,
\end{align*}
where we have used that $\varphi'' + 2 (1-\varphi^2) \varphi = 0$.
The proof of Theorem \ref{thm-conservation} is obtained from specific computations for the NLS equation (\ref{nls-evol}). 

\vspace{0.25cm}

\begin{proof1}{\em of Theorem \ref{thm-conservation}.}
	The energy functional $\hat{E}(v) : H^s(\R) \to \R$ is smooth in $v$ and $\bar{v}$ for $s \geq 1$. The evolution equation (\ref{nls-evol}) can be cast to the Hamiltonian form 
\begin{equation}
\label{Ham}
\frac{d v}{dt} = -i (1 - \epsilon^2 \partial_{\xi}^2)^{-1} \nabla_{\bar{v}} \hat{E}(v),
\end{equation}
where $\nabla_{\bar{v}} \hat{E}(v) = -F(v)$ is the variational derivative 
of $\hat{E}(v)$ with respect to $\bar{v}$. Since $(1 - \epsilon^2 \partial_{\xi}^2)^{-1} : L^2(\R) \to L^2(\R)$ is self-adjoint, the energy $\hat{E}(v)$ is constant in time $t \in [0,\tau_0)$ 
for the solution $v \in C^1([0,\tau_0),H^s(\R))$ with $s \geq 1$ which exists by Theorem \ref{thr-local-well}. 

The momentum functional $\hat{P}(v) : H^s(\R) \to \R$ is smooth in $v$ and $\bar{v}$ for $s \geq 2$. To prove its conservation, we consider the solution $v \in C^1([0,\tau_0),H^s(\R))$ for $s \geq 2$ which exists by Theorem \ref{thr-local-well}. Since $P(\varphi) = 0$, we can work equivalently 
with $P(u) = \hat{P}(v)$. Differentiating (\ref{nls}) in $\xi$, multiplying it by $\bar{u}$, adding the complex conjugate, and integrating in $\xi$ over $\RR$ yield:
\begin{align*}
& i \int_{\RR} \left[ \bar{u}(1-\epsilon^2 \partial_{\xi}^2) u_{t \xi} - u (1 - \epsilon^2 \partial_{\xi}^2) \bar{u}_{t \xi} \right] d\xi \\
& + \int_{\RR} 
\left[ \bar{u} u_{\xi \xi \xi} + u \bar{u}_{\xi \xi \xi} \right] d\xi 
+ 2 \int_{\RR} 
\left[ \bar{u} u_{\xi} + u \bar{u}_{\xi} - 3 |u|^2 (|u|^2)_{\xi}  \right] d\xi = 0.
\end{align*}
Multiplying (\ref{nls}) by $\bar{u}_{\xi}$, adding the complex conjugate, and integrating in $\xi$ over $\mathbb{R}$ yield 
\begin{align*}
& i \int_{\RR} \left[ \bar{u}_{\xi} (1-\epsilon^2 \partial_{\xi}^2) u_t - u_{\xi} (1 - \epsilon^2 \partial_{\xi}^2) \bar{u}_t \right] d\xi \\
& + \int_{\RR} 
\left[ \bar{u}_{\xi} u_{\xi \xi} + u_{\xi} \bar{u}_{\xi \xi} \right] d\xi 
+ 2 \int_{\RR} 
\left[ \bar{u}_{\xi} u + u_{\xi} \bar{u} - |u|^2 (|u|^2)_{\xi}  \right] d\xi = 0.
\end{align*}
Subtracting the two equations and integrating by parts give conservation of $P(u) = \hat{P}(v)$ in time $t \in [0,\tau_0)$. 
\end{proof1}

\begin{remark}	
	The mass functional $\hat{M}(v) : H^s(\R) \to \R$ is smooth in $v$ and $\bar{v}$ for $s \geq 1$ under the additional condition $v \in L^1(\R)$ 
	since $\varphi(\xi) \to \pm 1$ as $\xi \to \pm \infty$.
	Assuming the existence of the solution $v \in C^1([0,\tau_0),H^s(\R))$ for $s \geq 1$ such that $v(t,\cdot) \in L^1(\R)$ for $t \in [0,\tau_0)$, we can prove conservation of $\hat{M}(v)$ as follows. Multiplying (\ref{nls}) by $\bar{u}$, subtracting the complex conjugate, and integrating in $\xi$ over $\mathbb{R}$ yield 
	$$
	i \int_{\RR} \left[ \bar{u}(1-\epsilon^2 \partial_{\xi}^2) u_t + u (1 - \epsilon^2 \partial_{\xi}^2) \bar{u}_t \right] d\xi + \int_{\RR} 
	\left[ \bar{u} u_{\xi \xi} - u \bar{u}_{\xi \xi} \right] d\xi = 0.
	$$
	Integrating by parts gives conservation of $M(u)$ in tme $t \in [0,\tau_0)$.
\end{remark}

\section{Spectral stability and instability of the black soliton}
\label{sec-3}

In order to prove Theorem \ref{thm-spectral}, we first clarify the spectral properties of the Schr\"{o}dinger operator $L_+ : {\rm Dom}(L_+) \subset L^2(\R) \to L^2(\R)$, where $L_+ = -\partial_{\xi}^2 + 6 \varphi^2 - 2 = -\partial_{\xi}^2 + 4 - 6 \; {\rm sech}^2(\xi)$. As is well-known, any Schr\"{o}dinger operator with bounded potential can be extended to its form domain, hence we can write $L_+ : H^1(\R) \to H^{-1}(\R)$. Similarly, we can consider 
bounded operators $(1 - \epsilon^2 \partial_{\xi}^2)^{-1/2} : L^2(\R) \to H^1(\R)$ and $(1 - \epsilon^2 \partial_{\xi}^2)^{-1/2} : H^{-1}(\R) \to L^2(\R)$. By compositing the three operators above, we obtain a bounded operator 
$$
\mathcal{L}_+ = (1 - \epsilon^2 \partial_{\xi}^2)^{-1/2} L_+ (1 - \epsilon^2 \partial_{\xi}^2)^{-1/2} : L^2(\R) \to L^2(\R).
$$
Coercivity of $L_+ : H^1(\R) \to H^{-1}(\R)$ and 
$\mathcal{L}_+ : L^2(\R) \to L^2(\R)$ is obtained in the following lemma. 

\begin{lemma}
	\label{lem-L-plus}
	There exists $C > 0$ such that 
	\begin{equation}
	\label{coer-1}
\langle L_+ U, U \rangle \geq C \| U \|^2, \quad \mbox{\rm for every } U \in H^1(\R) : \;\; (U,\varphi') = 0
	\end{equation}
and
	\begin{equation}
	\label{coer-2}
	(\mathcal{L}_+ W,W) \geq C \| W \|^2, \quad \mbox{\rm for every } W \in L^2(\R) : \;\; (W,W_0) = 0,
	\end{equation}
where $\langle L_+ U, U \rangle$ is the dual action of $L_+ U \in H^{-1}(\R)$ on $U \in H^1(\R)$ and $W_0 := (1 - \epsilon^2   \partial_{\xi}^2)^{1/2} \varphi'$.
\end{lemma}

\begin{proof}
	We have $L_+ \varphi' = 0$ and $\varphi' \in H^1(\R)$ due to the translational invariance and the exponential decay of  $\varphi'(\xi) \to 0$ as $|\xi| \to \infty$. Since $\varphi'(\xi) > 0$ for all $\xi \in \R$, the zero eigenvalue of $L_+$ in $L^2(\R)$ is the lowest eigenvalue separated from the rest of the spectrum in $L^2(\R)$ by a gap. The spectral theorem implies the coercivity bound (\ref{coer-1}).
	
	By using the transformation $U =  (1 - \epsilon^2   \partial_{\xi}^2)^{-1/2} W$ with $W \in L^2(\R)$, we get $\mathcal{L}_+ W_0 = 0$ with $W_0 := (1-\epsilon \partial_{\xi}^2)^{1/2} \varphi' \in L^2(\R)$. The zero eigenvalue of $\mathcal{L}_+$ in $L^2(\R)$ is separated  from the continuous spectrum of $\mathcal{L}_+$ by a gap:
		$$
	\sigma_c(\mathcal{L}_+) = \left\{ \frac{4+k^2}{1+\epsilon^2 k^2}, \;\; k \in \mathbb{R} \right\} = 
	\left\{ \begin{array}{ll} 
	& [4,\epsilon^{-2}], \quad \epsilon \in (0,\frac{1}{2}), \\
	& [ \epsilon^{-2},4], \quad \epsilon \in (\frac{1}{2},\infty) \end{array} \right.
	$$
The coercivity bound (\ref{coer-1}) in $L^2(\R)$ implies a coercivity bound in $H^1(\R)$, that is, 
\begin{equation*}
(L_+ U, U) \geq C \| U \|_{\epsilon}^2, \quad \mbox{\rm for every } U \in H^1(\R) : \;\; (U,\varphi') = 0.
\end{equation*}
This bound is equivalent to (\ref{coer-2}) since 
$(L_+ U, U) = (\mathcal{L}_+ W, W)$, $\| U \|_{\epsilon}^2 = \| W \|^2$, and $(U,\varphi') = (W, W_0)$ for $U = (1-\epsilon^2 \partial_{\xi}^2)^{-1/2} W$.
\end{proof}

By Lemma \ref{lem-L-plus}, we can define the constrained operator
\begin{align*}
\mathcal{T}_+ := \mathcal{L}_+ |_{\{W_0 \}^{\perp}}  \; : \; L^2(\R) |_{\{W_0 \}^{\perp}} \mapsto L^2(\R) |_{\{W_0\}^{\perp}}, 
\end{align*}
Thanks to the bound (\ref{coer-2}), $\mathcal{T}_+$ is invertible and strictly positive with a bounded and strictly positive inverse $\mathcal{T}_+^{-1}$. 

In addition to Lemma \ref{lem-L-plus}, we also need a technical computation 
related to the Schr\"{o}dinger operator $L_- : {\rm Dom}(L_-) \subset L^2(\R) \to L^2(\R)$, where $L_- = -\partial_{\xi}^2 + 2 \varphi^2 - 2 = -\partial_{\xi}^2 - 2 \; {\rm sech}^2(\xi)$. The technical computation is given by the following lemma. 

\begin{lemma}
	\label{lem-L-minus} The linear inhomogeneous equation 
\begin{equation}
\label{lin-inhom}
	L_- V_{\varphi} =  (1 - \epsilon^2 \partial_{\xi}^2) \varphi'
\end{equation}
	admits a unique even and bounded solution $V_{\varphi}$ satisfying  
	\begin{align*}
	(\varphi',V_{\varphi})_{\epsilon} = -1 + \frac{8}{5} \epsilon^4.
	\end{align*}
Hence $(\varphi',V_{\varphi})_{\epsilon} \leq 0$ if $\epsilon \in [0,\epsilon_0]$ and $(\varphi',V_{\varphi})_{\epsilon} > 0$ if $\epsilon \in (\epsilon_0,\infty)$, where $\epsilon_0 := (5/8)^{1/4}$.
\end{lemma}

\begin{proof}
A general solution $V_{\varphi}$ of the linear equation (\ref{lin-inhom}) 
can be found in the explicit form by substitutions as 
	\begin{equation*}
	V_{\varphi}(\xi) = -\frac{1}{2} (1 + 2 \epsilon^2) + \frac{3}{2} \epsilon^2 {\rm sech}^2(\xi) + c_1 \tanh(\xi) + c_2 \left[ \xi \tanh(\xi) - 1 \right],
	\end{equation*}
	where $c_1$ and $c_2$ are arbitrary constants. If $V_{\varphi}$ is required to be even and bounded, then $c_1 = 0$ and $c_2 = 0$ respectively.
	Although $V_{\varphi} \notin L^2(\R)$, the inner product $(\varphi',V_{\varphi})_{\epsilon}$ makes sense due to the exponential 
	decay of $\varphi'(\xi) \to 0$ as $|\xi| \to \infty$ and can be computed explicitly:
\begin{align*}
	(\varphi',V_{\varphi})_{\epsilon} &= \int_{\mathbb{R}} \left[ 
	(1-4 \epsilon^2) {\rm sech}^2(\xi) + 6 \epsilon^2 {\rm sech}^4(\xi) \right] 
	\left[ -\frac{1}{2} (1 + 2 \epsilon^2) + \frac{3}{2} \epsilon^2 {\rm sech}^2(\xi) \right] d \xi \\
	&= -(1+2\epsilon^2) (1 - 4 \epsilon^2) + 2 \epsilon^2 (1-4 \epsilon^2) 
	-4 \epsilon^2 (1+2\epsilon^2) + \frac{48}{5} \epsilon^4 \\
	&= -1 + \frac{8}{5} \epsilon^4,
	\end{align*}
which yields the result.
\end{proof}

We are now ready to prove Theorem \ref{thm-spectral}.

\vspace{0.25cm}

\begin{proof1}{\em of Theorem \ref{thm-spectral}.}
By Fredholm's theory, if there exists a solution of the spectral problem (\ref{lin-stab}) in $H^1(\R) \times H^1(\R)$ for $\lambda \neq 0$, then the component $V$ satisfies the orthogonality condition $(\varphi',V)_{\epsilon} = 0$. We will show that an eigenvalue $\lambda_0 \in \mathbb{C} \backslash i \mathbb{R}$ exists if and only if $\epsilon \in (\epsilon_0,\infty)$. Since, 
for every eigenvalue $\lambda_0$, $-\lambda_0$ is also an eigenvalue, this will prove the theorem.

For any eigenvalue $\lambda_0 \neq 0$, the corresponding eigenvector 
$(U,V) \in H^1(\R) \times H^1(\R)$ must satisfy $(\varphi',V)_{\epsilon} = 0$.  The second equation of the system (\ref{lin-stab}) can be written in the form
\begin{equation}
\label{tech-comp}
\mathcal{L}_+ (1-\epsilon^2 \partial_{\xi}^2)^{1/2} U = -\lambda_0 (1-\epsilon^2 \partial_{\xi}^2)^{1/2} V,
\end{equation}
where the right-hand side is orthogonal to $W_0 = (1 - \epsilon^2   \partial_{\xi}^2)^{1/2} \varphi'$ since
$$
((1-\epsilon^2 \partial_{\xi}^2)^{1/2} V, W_0) = 
((1-\epsilon^2 \partial_{\xi}^2)^{1/2} V, (1-\epsilon^2 \partial_{\xi}^2)^{1/2} \varphi') = (V, \varphi')_{\epsilon} = 0.
$$
Furthermore, we have 
$$
((1-\epsilon^2 \partial_{\xi}^2)^{1/2} U, W_0) = 
((1-\epsilon^2 \partial_{\xi}^2)^{1/2} U, (1-\epsilon^2 \partial_{\xi}^2)^{1/2} \varphi') = (U, \varphi')_{\epsilon}.
$$
One can select the eigenfunction uniquely by requiring that $(U, \varphi')_{\epsilon} = 0$ after adding a multiple of $\varphi'$ to $U$. 
With this convention, $\mathcal{L}_+$ in (\ref{tech-comp}) can be replaced by $\mathcal{T}_+$ so that equation (\ref{tech-comp}) can be solved in the form
\begin{equation*}
(1-\epsilon^2 \partial_{\xi}^2)^{1/2} U = -\lambda_0 \mathcal{T}_+^{-1} (1-\epsilon^2 \partial_{\xi}^2)^{1/2} V.
\end{equation*}
Substituting this formula into the first equation of the system (\ref{lin-stab}) yields
\begin{equation*}
L_- V = -\lambda_0^2 (1-\epsilon^2 \partial_{\xi}^2)^{1/2} \mathcal{T}_+^{-1} (1-\epsilon^2 \partial_{\xi}^2)^{1/2} V.
\end{equation*}
Since this eigenvalue problem is self-adjoint and  $(1-\epsilon^2 \partial_{\xi}^2)^{1/2} \mathcal{T}_+^{-1} (1-\epsilon^2 \partial_{\xi}^2)^{1/2} : H^1(\R) \to H^{-1}(\R)$ is a strictly positive operator by Lemma \ref{lem-L-plus}, we obtain $\lambda_0^2 \in \mathbb{R}$, i.e. $\lambda_0 \in \R\cup i \R$. Therefore, $\lambda_0$ is an eigenvalue 
in $\mathbb{C} \backslash i \R$ if and only if $-\lambda_0^2 < 0$. 
Consequently, such an eigenvalue exists if and only if 
\begin{equation}
\label{Rayleigh}
\inf_{{\tiny \begin{array}{l} V \in H^1_{\epsilon}(\R) \backslash \{0\} \\ (\varphi',V)_{\epsilon} = 0 \end{array}}} 
	\frac{\langle L_- V, V \rangle}{\langle (1-\epsilon^2 \partial_{\xi}^2)^{1/2} \mathcal{T}_+^{-1} (1-\epsilon^2 \partial_{\xi}^2)^{1/2} V, V \rangle} < 0.
\end{equation} 
Since the denominator is strictly positive by Lemma \ref{lem-L-plus}, 
the sign of the left-hand side of (\ref{Rayleigh}) is determined by the sign of $-\mu_0^2$ in a simpler variational problem
	\begin{equation}
\label{Rayleigh-revised}
-\mu_0^2 = \inf_{{\tiny \begin{array}{l} V \in H^1_{\epsilon}(\R) \backslash \{0\} \\ (\varphi',V)_{\epsilon} = 0 \end{array}}} 
\frac{\langle L_- V, V \rangle}{\| V \|^2}.
\end{equation}
As is shown in \cite{GalloMenza}, the sign of $-\mu_0^2$ in (\ref{Rayleigh-revised}) depends on the sign of 
\begin{align*}
\lim_{\lambda \to 0^-} ((L_- - \lambda I)^{-1} (1-\epsilon^2 \partial_{\xi}^2) \varphi', (1-\epsilon^2 \partial_{\xi}^2) \varphi') 
& = 
\lim_{\lambda \to 0^-} ((L_- - \lambda I)^{-1} L_- V_{\varphi}, (1-\epsilon^2 \partial_{\xi}^2) \varphi') \\
& = (V_{\varphi},\varphi')_{\epsilon},
\end{align*}
which changes the sign if $\epsilon = \epsilon_0$. The following dichotomy exists 
(see \cite[Theorem 1.1]{GalloMenza}):
	\begin{itemize}
		\item[(i)] $-\mu_0^2 \geq 0$ if $(\varphi',V_{\varphi})_{\epsilon} \leq 0$,
		\item[(ii)] $-\mu_0^2 < 0$ if $(\varphi',V_{\varphi})_{\epsilon} > 0$.
	\end{itemize}
In case (i), the condition (\ref{Rayleigh}) is not satisfied so that no isolated eigenvalue $\lambda_0$ with ${\rm Re}(\lambda_0) > 0$ of the spectral problem (\ref{lin-stab}) exists. This is the spectral stability case which corresponds to $\epsilon \in (0,\epsilon_0]$. In case (ii), the condition (\ref{Rayleigh}) is satisfied 
so that there exists an isolated positive eigenvalue $\lambda_0 \in \R$. 
This is the spectral instability case which corresponds to $\epsilon \in (\epsilon_0,\infty)$.
\end{proof1}

\section{Numerical illustrations}
\label{sec-4}

We approximate solutions of the NLS equation (\ref{nls}) numerically by using a finite-difference method. The line $\mathbb{R}$ is truncated on the symmetric interval $[-L,L]$ for sufficiently large $L > 0$ subject to the Neumann boundary conditions $u_{\xi}(t,-L) = u_{\xi}(t,L) = 0$ for every $t > 0$.
One can also use the inhomogeneous Dirichlet conditions $u(t,\pm L) = \pm 1$ 
as an alternative truncation, which we did not explore in our numerical computations.

Representing the solution $u(t,\xi)$ as a column vector ${\bf u}(t)$ on an equally spaced grid of $2K+1$ grid points on the interval $[-L,L]$ yields the evolutionary problem in the form 
\begin{equation}
\label{nls-discrete}
i (1 - \epsilon^2 A_N) \frac{d {\bf u}}{dt} + A_N {\bf u} + 2 (1 - |{\bf u}|^2) {\bf u} = 0, 
\end{equation}
where ${\bf u}(t) : \mathbb{R} \to \mathbb{C}^{2K+1}$ and $A_N \in \mathbb{M}^{(2K+1) \times (2K+1)}$ is the matrix approximation of the central difference for the second spatial derivative which incorporates the Neumann boundary conditions at the end points,
\begin{align*}
(A_N {\bf u})_k &= \frac{u_{k+1} - 2 u_k + u_{k-1}}{h^2}, \qquad \qquad 2 \leq k \leq 2K, 
\end{align*}
and
\begin{align*}
(A_N {\bf u})_1 &= \frac{2(u_{2} - u_1)}{h^2}, \quad 
(A_N {\bf u})_{2K+1} = \frac{2(u_{2K} - u_{2K+1})}{h^2}.
\end{align*}
As is well-known, the discretization error of the central difference has the order of $\mathcal{O}(h^2)$. 

Iterations in time are performed with the Crank--Nicholson method over an equally spaced temporal grid with the time step $\tau$. The numerical approximation of ${\bf u}(t)$ at the time level $t_m = m \tau$ is denoted by ${\bf u}^{(m)}$. 
The Crank--Nicolson method is given by the iterative rule: 
\begin{align}
\label{nls-crank}
& \left(1 - \epsilon^2 A_N - \frac{i \tau}{2} A_N - i \tau (1-|{\bf u}^{(m+1)}|^2) \right) {\bf u}^{(m+1)} \notag \\
& =  \left(1 - \epsilon^2 A_N + \frac{i \tau}{2} A_N + i \tau (1-|{\bf u}^{(m)}|^2) \right) {\bf u}^{(m)}, 
\end{align}
for integer $m \geq 0$. As is also well-known, the discretization error of the Crank--Nicolson method has the global error of $\mathcal{O}(\tau^2)$ and the stability of iterations is unconditional with respect to the time step $\tau$ relative to $h$.

\begin{figure}[htb!]
	\includegraphics[width=8cm,height=6cm]{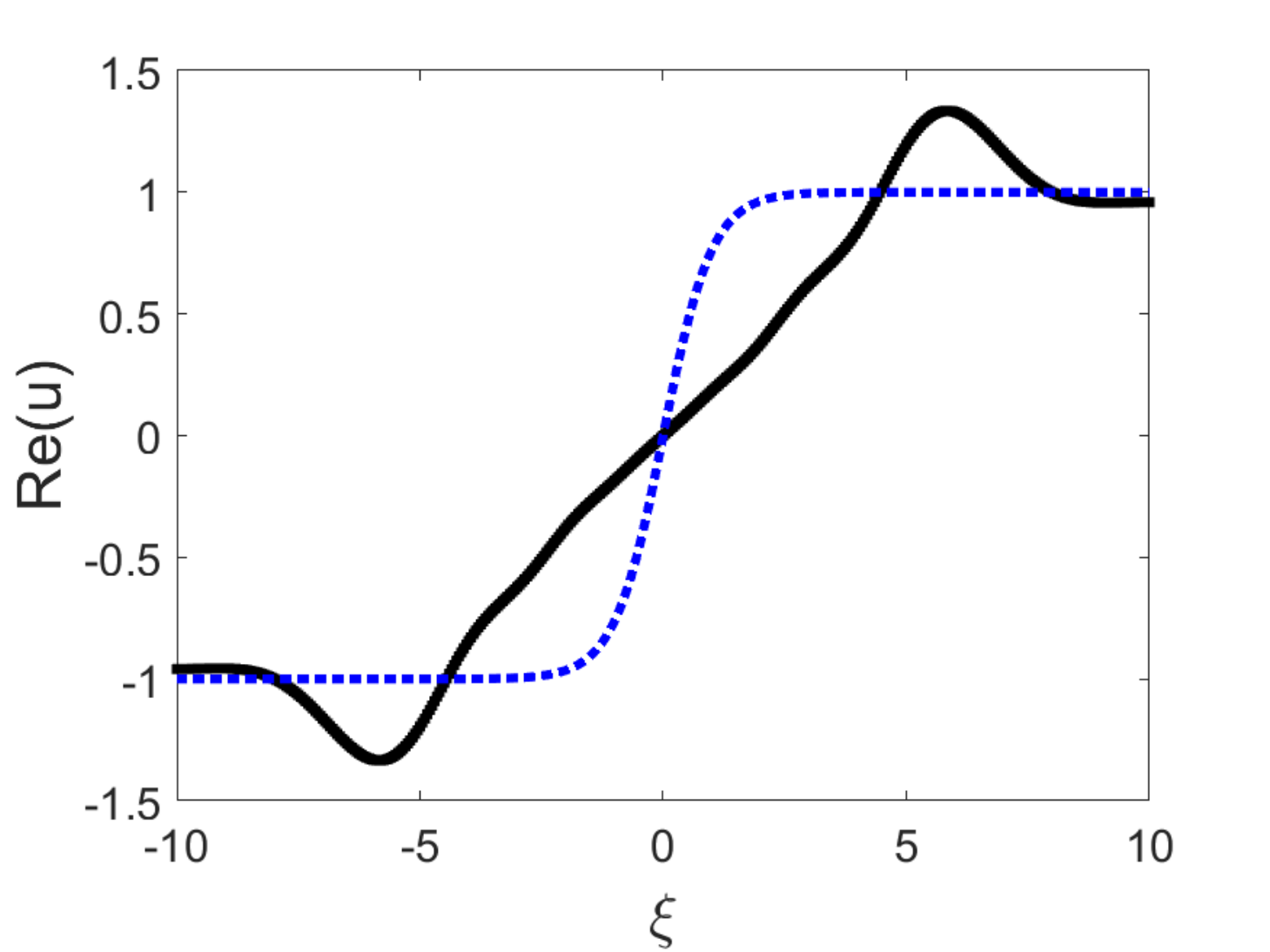}
	\includegraphics[width=8cm,height=6cm]{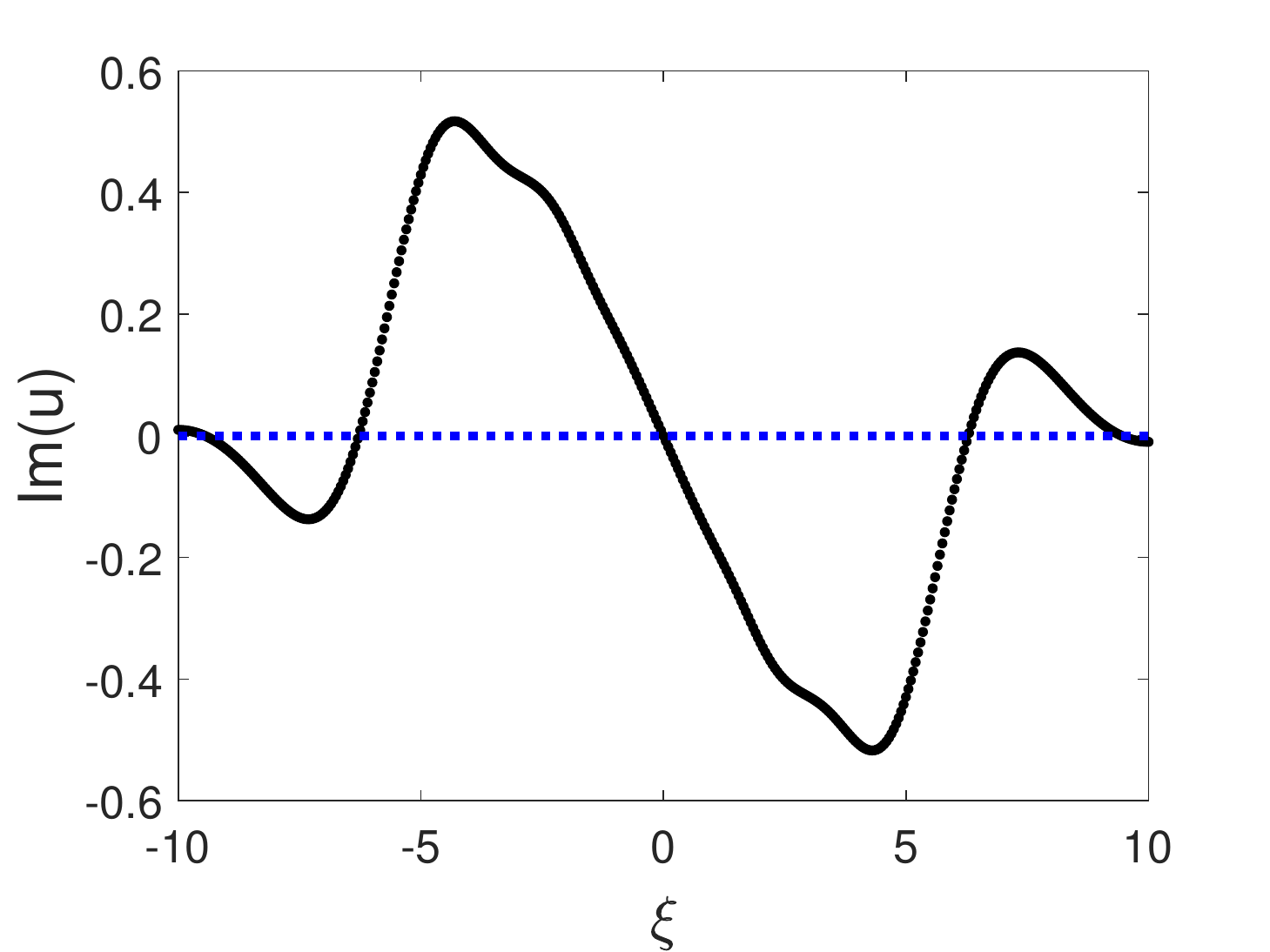}
	\caption{Real (left) and imaginary (right) parts of the solution 
		at time $t = 0$ (blue dots) and time $t = 5$ (black line).}
	\label{fig-Simulation}
\end{figure}

Before reporting outcomes of our numerical computations, we would like to confirm the order and unconditional stability of the numerical method by comparing computations with exact solutions. To do so, we consider the linear Schr\"{o}dinger evolution according to the iterations of the linear system:
\begin{equation}
\label{nls-crank-lin}
\left(1 - \epsilon^2 A_N - \frac{i \tau}{2} A_N \right) {\bf u}^{(m+1)} =  \left(1 - \epsilon^2 A_N + \frac{i \tau}{2} A_N \right) {\bf u}^{(m)}.
\end{equation}
The iterative scheme is implicit but invertion of the matrix on the left side is independent of the time step $m$, so that the iterative scheme is implemented as the explicit step:
\begin{equation*}
{\bf u}^{(m+1)} =  \left(1 - \epsilon^2 A_N - \frac{i \tau}{2} A_N \right)^{-1} \left(1 - \epsilon^2 A_N + \frac{i \tau}{2} A_N \right) {\bf u}^{(m)}.
\end{equation*}
The initial condition was ${\bf u}^{(0)} = \tanh(\xi_k)$, where $\xi_k = (k-1-K) h$ with $h = L/K$. We have used $L = 10$ and $K = 200$ for computations. The outcomes of the numerical simulations is shown for $\epsilon = 0.5$ on Figure \ref{fig-Simulation}. The real (left) and imaginary (right) parts of the numerical approximation of the solution ${\bf u}$ are shown for the initial data $t = 0$ (blue dots) and for the final time $t = 5$ (black lines). The simulations show 
that the black soliton deteriorates in the time evolution because of the linear dispersion.

The actual dynamics of the linear system (\ref{nls-crank-lin}) is not  important for our study. It is used here to estimate the computational error since the initial-boundary value problem for the linear Schr\"{o}dinger equation 
\begin{equation}
\label{lin-Schr}
\left\{ \begin{array}{lll} i (1 - \epsilon^2 \partial_{\xi}^2) u_t + u_{\xi \xi} = 0, \quad & x \in (-L,L), \quad & t > 0, \\
u_{\xi}(t,-L) = 0, \;\; u_{\xi}(t,L) = 0,  \quad & & t > 0, \\
u(0,\xi) = \tanh(\xi), \quad & x \in [-L,L], \quad &
\end{array} \right.
\end{equation}
can be solved with the separation of variables and Fourier series. The exact solution is given by the Fourier cosine series:
\begin{equation}
\label{Fourier}
u(t,\xi) = \frac{1}{2} a_0 + \sum_{n=1}^{\infty} a_n \exp\left(-\frac{i k_n^2 t}{1 + \epsilon^2 k_n^2}\right) \cos(k_n \xi + k_n L),
\end{equation}
where $k_n := \frac{\pi n}{2L}$, $n \in \mathbb{N}$ and the Fourier coefficients are computed from the initial data as
\begin{align*}
a_n = \frac{1}{L} \int_{-L}^L \tanh(\xi) \cos(k_n \xi + k_n L) d\xi.
\end{align*}

By using the Fourier interpolation with $2K+1$ grid points based on the Fourier cosine series (\ref{Fourier}), we obtain another numerical approxiation of the solution $u = u(t,\xi)$, whose truncation error is exponentially small with respect to the spacing $h$ between the grid points. Hence, the maximal distance between the two numerical solutions at time $t$ can be considered as the error function of the finite-difference approximation. Figure \ref{fig-Error} shows the error function versus time for computations with 
two step sizes $h = L/K$ for $K = 200$ and $K = 400$, whereas the time step is adjusted to be  $\tau = h$. The ratio of the two errors is always close to $4.0$ which confirms the second-order accuracy of the finite-difference method.

\begin{figure}[htb!]
	\includegraphics[width=14cm,height=8cm]{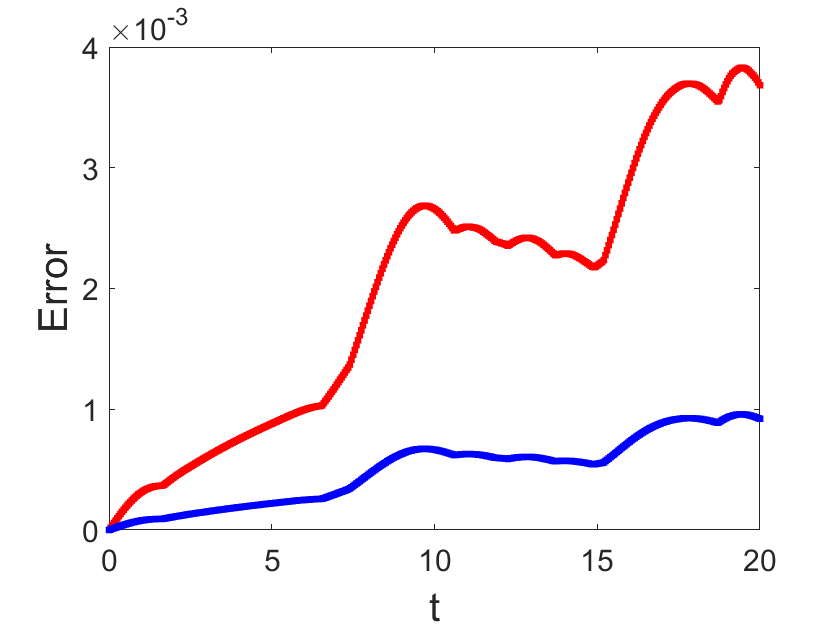}
	\caption{The error function versus time for computations with 
		$h = L/K$ and $\tau = h$ for $K = 200$ (red line) and $K = 400$ (blue line).}
	\label{fig-Error}
\end{figure} 

\begin{figure}[htb!]
	\includegraphics[width=8cm,height=6cm]{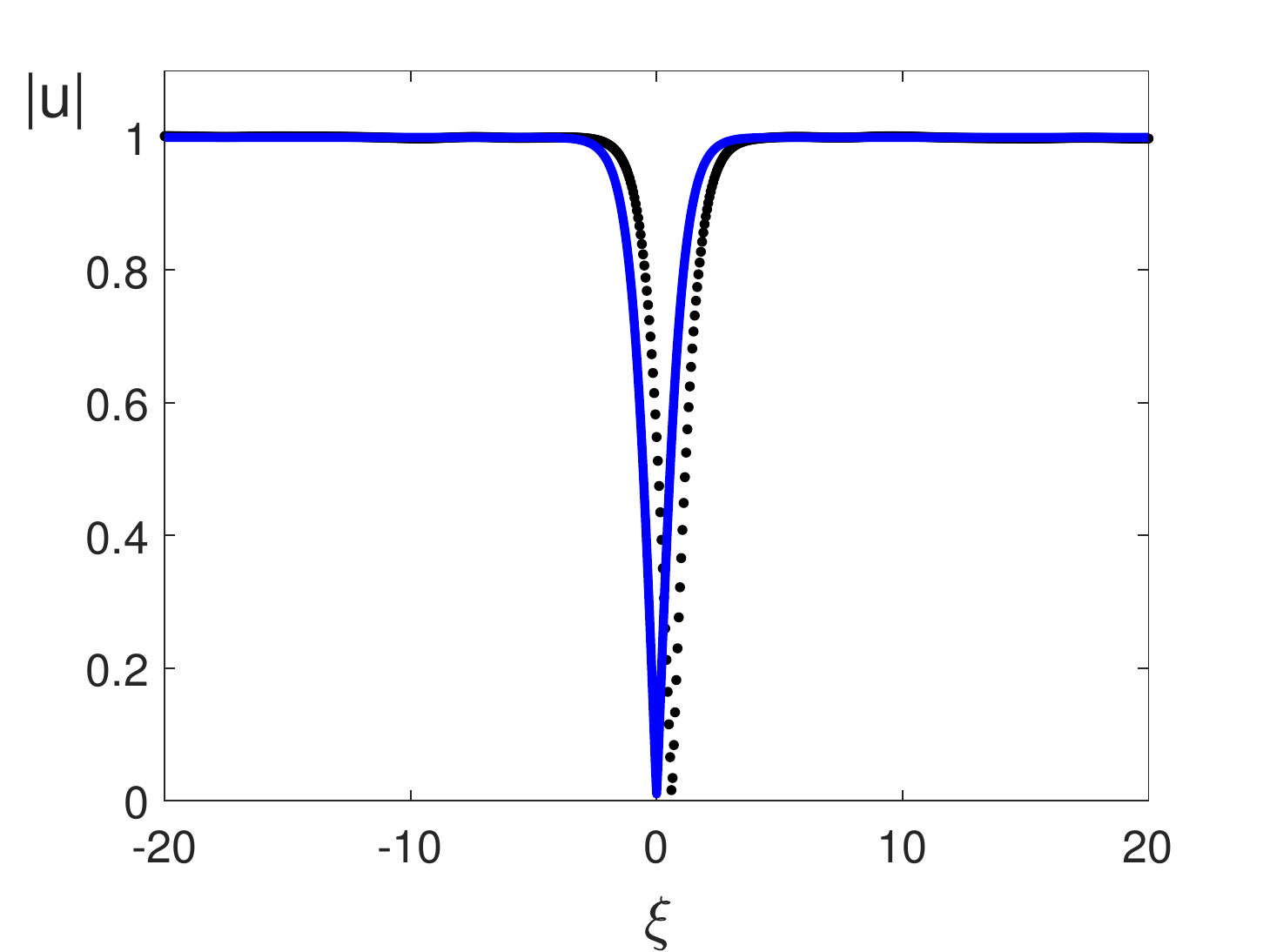}
	\includegraphics[width=8cm,height=6cm]{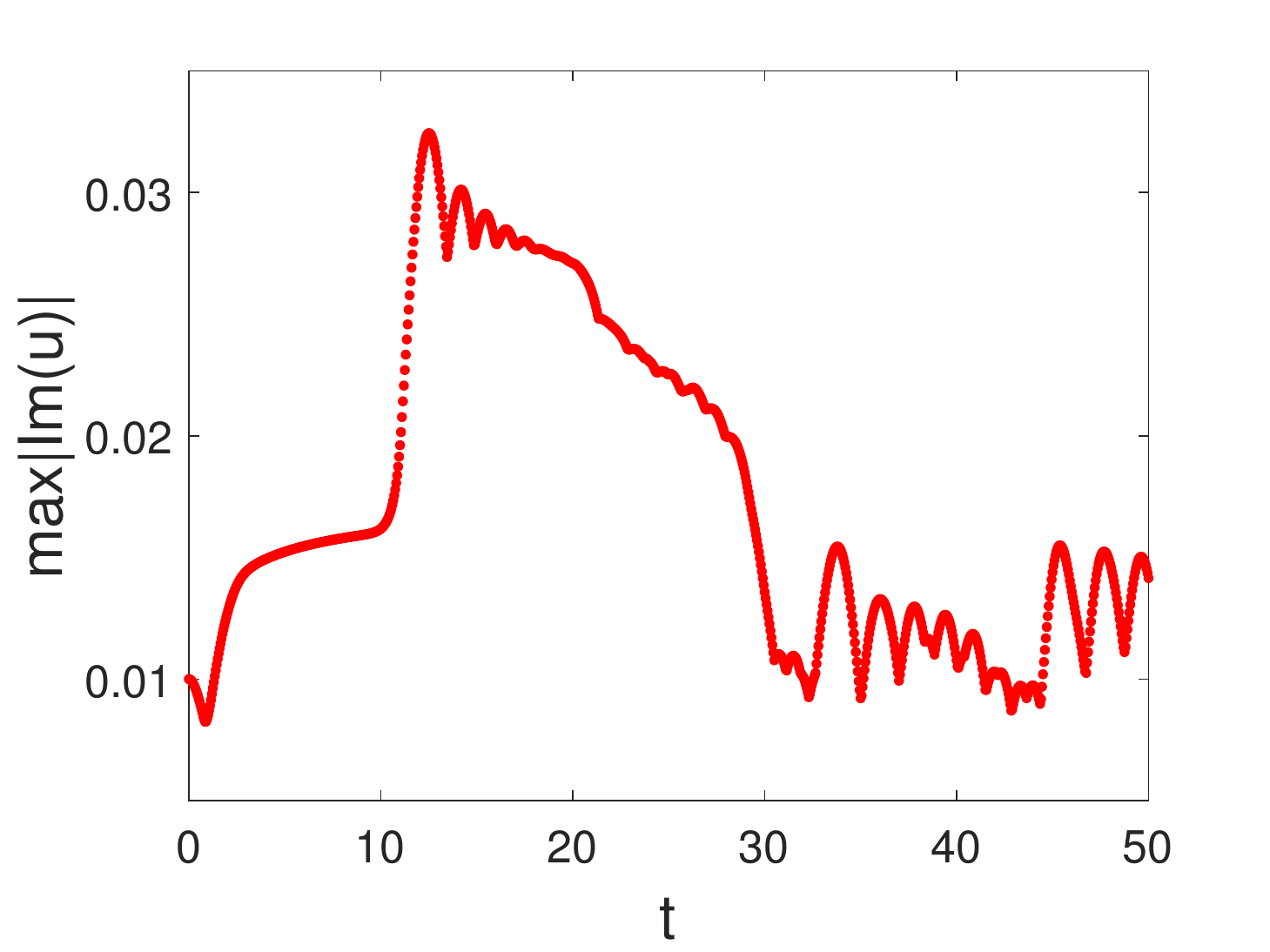}\\
	\includegraphics[width=14cm,height=8cm]{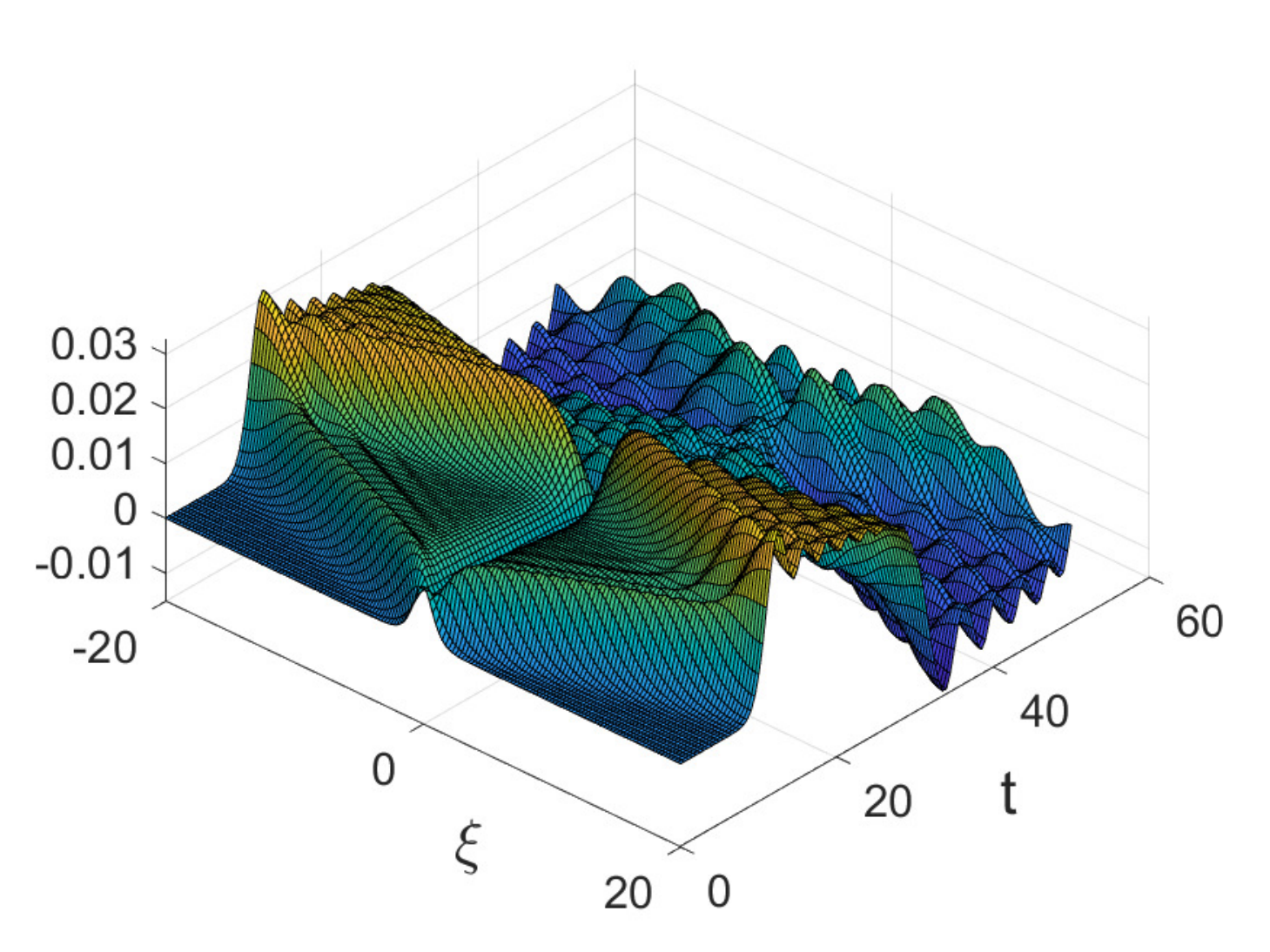}
	\caption{Numerical approximation of the dynamics of the perturbed black soliton for $\epsilon = 0.5$. Top left: the profile of $|u|$ versus $\xi$ for time $t = 0$ (blue line) and time $t = 50$ (black dots). Top right: the maximal value of $|{\rm Im}(u)|$ versus $t$. Bottom: the solution surface for ${\rm Im}(u)$ on the $(t,\xi)$ plane. }
	\label{fig-05}
\end{figure} 

Having tested the finite-difference method on the linear system (\ref{nls-crank-lin}), we can now perform iterations of the nonlinear system 
(\ref{nls-crank}). The iterative scheme is fully implicit. In order to make it 
explicit, we invert the matrix in the left side at each iteration $m$ 
by using the previous value of $|{\bf u}^{(m)}|^2$,
\begin{equation*}
\tilde{\bf u}^{(m+1)} = \left(1 - \epsilon^2 A_N - \frac{i \tau}{2} A_N - i \tau (1-|{\bf u}^{(m)}|^2) \right)^{-1} \left(1 - \epsilon^2 A_N + \frac{i \tau}{2} A_N + i \tau (1-|{\bf u}^{(m)}|^2) \right) {\bf u}^{(m)}, 
\end{equation*}
and use Heun's predictor--corrector method for ${\bf u}^{(m+1)}$ to restore the second-order 
accuracy of the time iterations. The initial data was chosen as 
\begin{align*}
{\bf u}^{(0)} = \tanh(\xi_k) + i a {\rm sech}^2(\xi_k), 
\end{align*}
where $a > 0$ is the amplitude factor for the perturbation to the black soliton. The perturbation is needed to observe the stable versus unstable dynamics of the perturbed black soliton in the time evolution. We have chosen $L = 20$, $K = 400$, and $a = 0.01$.

Figure \ref{fig-05} shows the outcomes of the numerical simulations 
for $\epsilon = 0.5$. According to Theorem \ref{thm-spectral}, the black 
soliton is spectrally stable for this value of $\epsilon$ since $\epsilon_0 = (5/8)^{1/4} \approx 0.89$. Indeed, we observe that the initial perturbation pushes the black soliton to the right for a small distance (top left panel) but the fluctuation of the imaginary part of the solution are bounded in the time evolution (top right panel). The solution surface for the imaginary part (bottom panel) shows that the perturbations to the black soliton are pushed towards the boundaries during the time evolution where they are reflected due to the Neumann boundary conditions. The perturbed black soliton preserves its shape in the case $\epsilon = 0.5$.

In comparison with the stable dynamics of the perturbed black soliton 
for $\epsilon = 0.5 < \epsilon_0$, Figure \ref{fig-1} shows the unstable dynamics for $\epsilon = 1 > \epsilon_0$. The perturbations to the black soliton in ${\rm Im}(u)$ grow from the initial value of $a = 0.01$ towards the unit magnitude. As a result, the black soliton 
is completely destroyed in the time evolution and the final profile of $|u|$ versus $\xi$ for $t = 10$ (black dots on the top left panel) shows non-solitonic solutions.

\begin{figure}[htb!]
	\includegraphics[width=8cm,height=6cm]{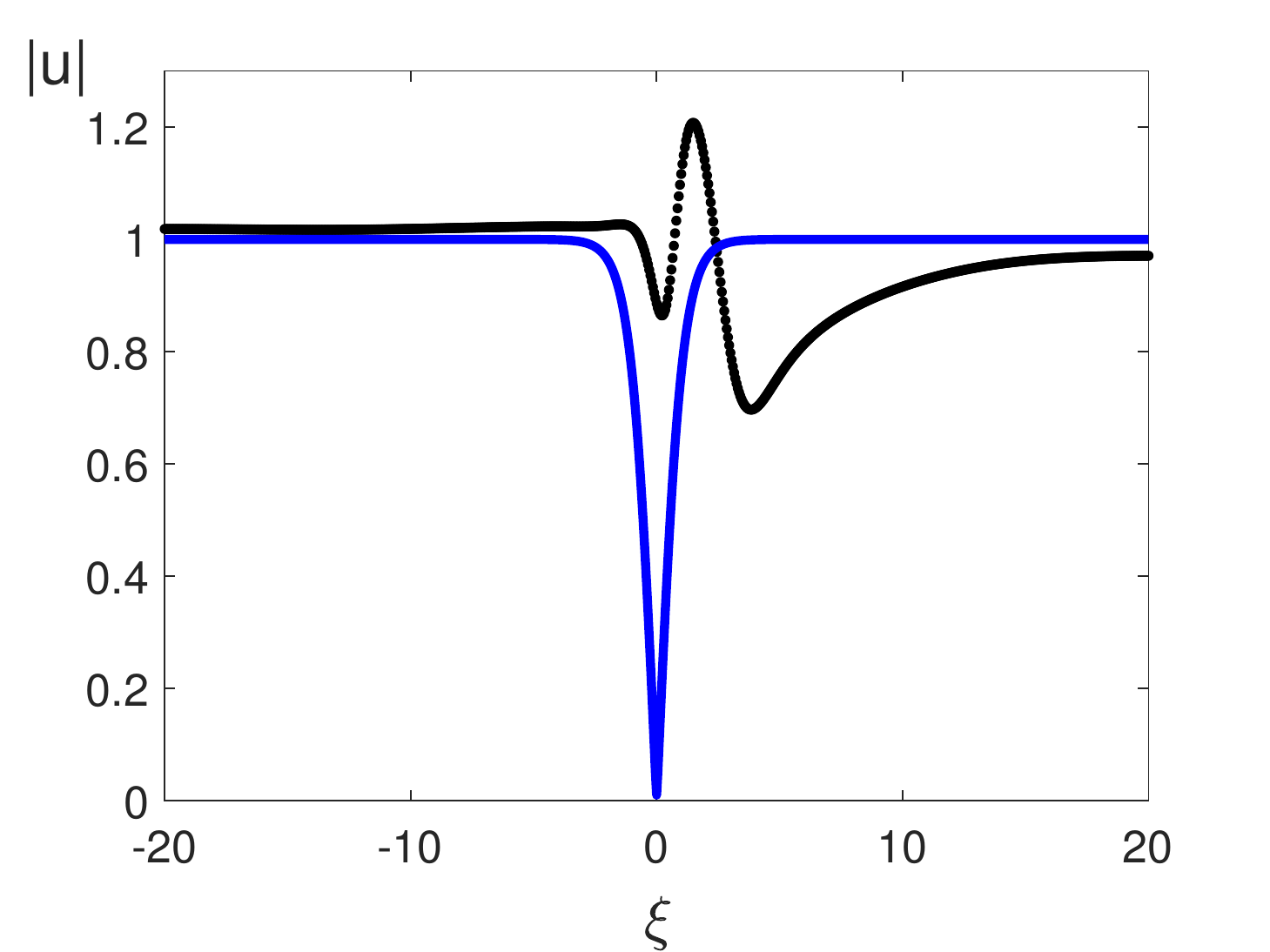}
	\includegraphics[width=8cm,height=6cm]{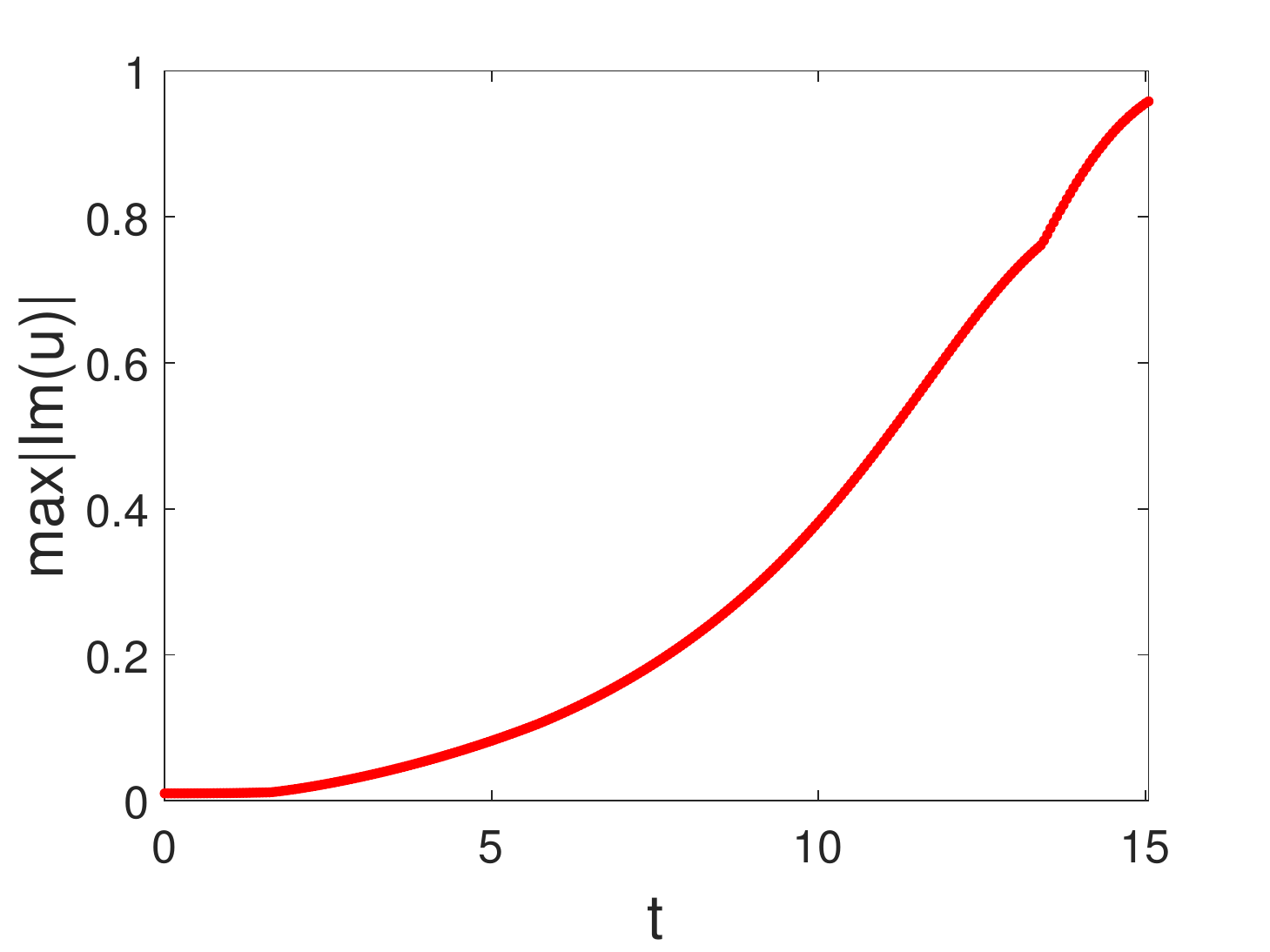}\\
	\includegraphics[width=14cm,height=8cm]{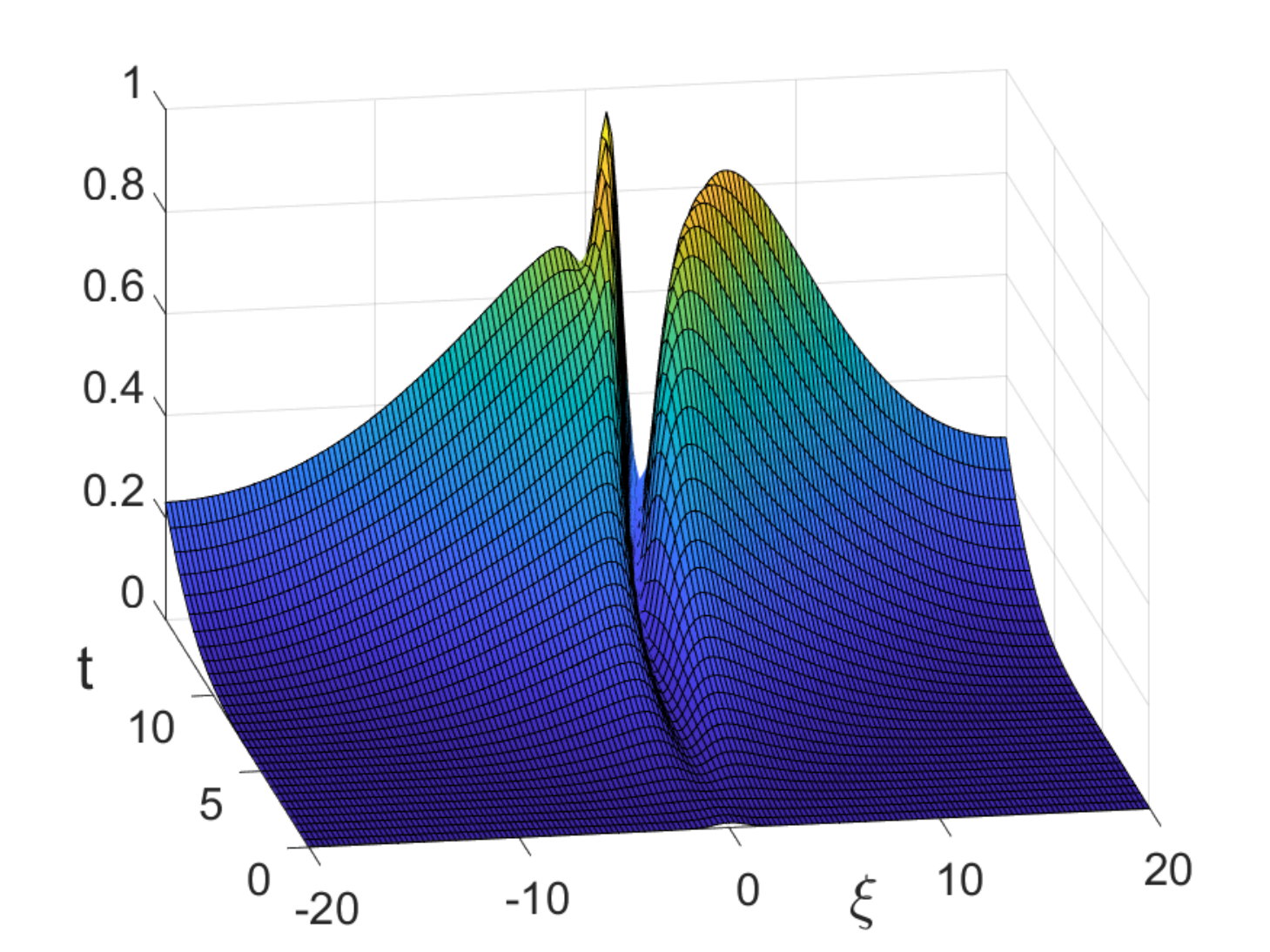}
	\caption{The same as on Figure \ref{fig-05} but for $\epsilon = 1$.}
	\label{fig-1}
\end{figure}

We have performed computations for $\epsilon$ closer to $\epsilon_0$ and observed the same stable and unstable dynamics of the perturbed black soliton similar to Figures \ref{fig-05} and \ref{fig-1}. 
The actual value of the instability threshold depends generally 
on the half-length $L$ of the truncated interval $[-L,L]$.

Finally, we can also inspect the time evolution of the perturbed black soliton 
in the canonical defocusing NLS equation (\ref{nls-canon}) which corresponds 
to the NLS model (\ref{nls}) with $\epsilon = 0$. Figure \ref{fig-0} show 
that the perturbed black soliton is stable in the time evolution  but the perturbations quickly become noisy in the time evolution due to multiple reflections from the boundaries. This dynamics agrees well with the property 
of the NLS equation (\ref{nls-canon}) that 
the imaginary part of the perturbation is not controlled in the Sobolev space 
of $H^1(\mathbb{R})$ since the energy and momentum are not coercive in $H^1(\mathbb{R})$. In comparison, perturbations of the black soliton in the regularized NLS equation (\ref{nls}) with $\epsilon > 0$ are well-defined in the space $H^1(\R)$ as a continuously differentiable function of time. Hence, the imaginary part of the perturbations on 
Figures \ref{fig-05} and \ref{fig-1} are much smoother and decaying in the spatial coordinate compared to the one on Figure \ref{fig-0}.

\begin{figure}[htb!]
	\includegraphics[width=8cm,height=6cm]{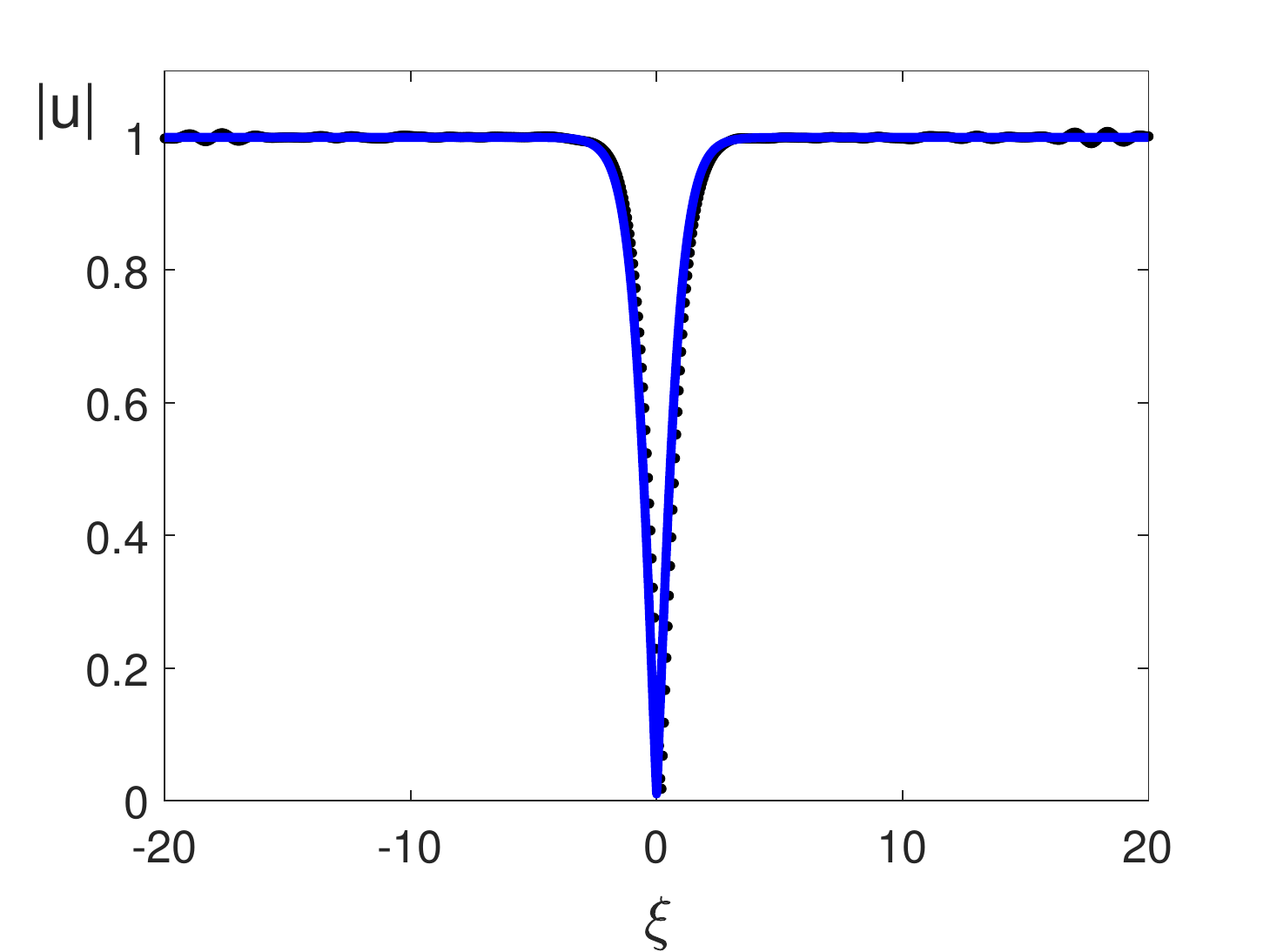}
	\includegraphics[width=8cm,height=6cm]{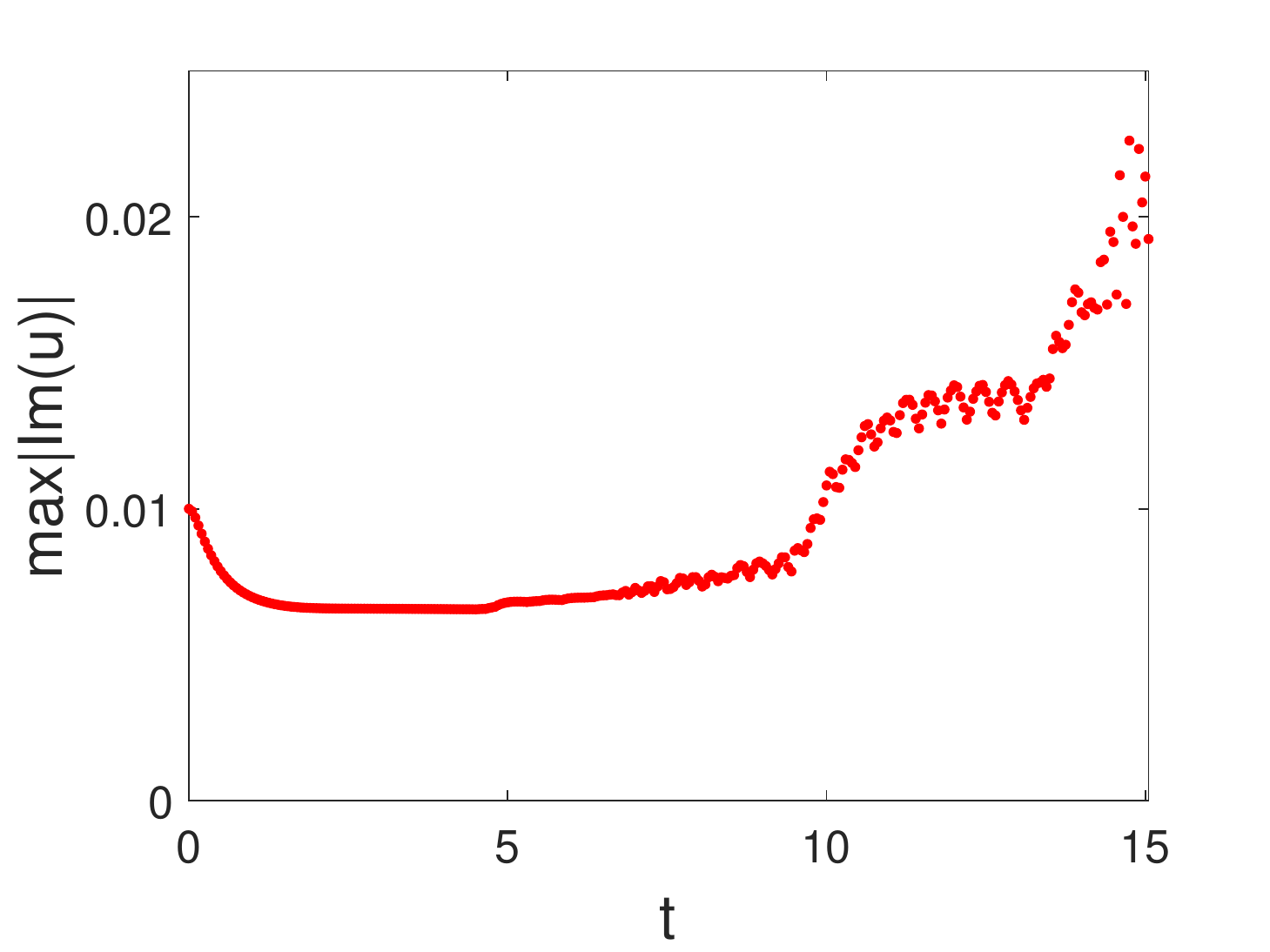}\\
	\includegraphics[width=14cm,height=8cm]{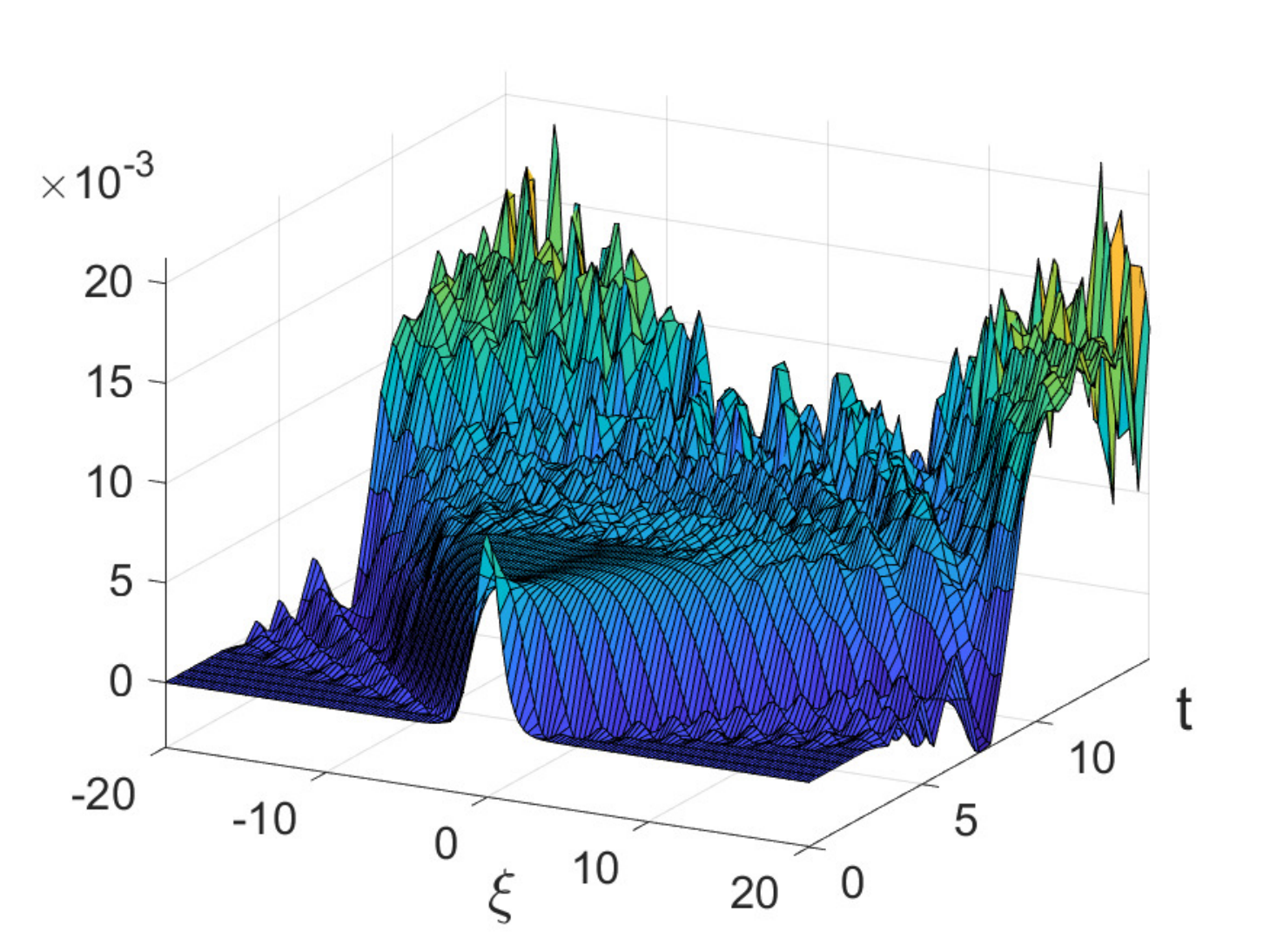}
	\caption{The same as on Figure \ref{fig-05} but for $\epsilon = 0$.}
	\label{fig-0}
\end{figure}

\section{Discussion}
\label{sec-5}

We have shown analytically and illustrated numerically that the regularized 
NLS equation (\ref{nls}) admits smooth time-dependent solutions near the black soliton with perturbations defined as $C^1([0,\tau_0), H^s(\R))$ with $s > \frac{1}{2}$. We have shown the spectral stability of the black soliton
for $\epsilon \leq \epsilon_0 := (5/8)^{1/4}$ and spectral instability for $\epsilon > \epsilon_0$. The question of orbital stability of the black soliton is left open since the energy conservation can be used to control a weighted $H^1(\R)$ norm of the perturbation, whereas the momentum conservation is only defined for perturbations in $H^s(\R)$ with $s \geq 2$.

Among further directions of the research, one can consider further generalizations of the regularized NLS equation with cubic nonlinearity in the form 
\begin{equation}
\label{nls-combined}
i (1 - \epsilon^2 \partial_x^2  - \delta^2 |u|^2) u_t + u_{xx} + 
2 (1-|u|^2) u = 0,
\end{equation}
with two parameters $\epsilon \geq 0$ and $\delta \in [0,1]$. The case $\epsilon = 0$ and $\delta = 1$ was considered in \cite{Plum} as the NLS model with intensity-dependent dispersion, 
\begin{equation}
\label{nls-idd}
i (1 - |u|^2) u_t + u_{xx} + 
2 (1-|u|^2) u = 0.
\end{equation}
We proved in \cite{Plum} that the perturbations to the black soliton are controlled in a weighted $H^1(\R)$ space from the energy and momentum conservation of the NLS model (\ref{nls-idd}). Spectral stability of the black soliton was also proven by characterizing the purely discrete spectrum of the spectral stability problem. The question of local well-posedness for the perturbations of the black soliton was left open since the black soliton satisfies the boundary conditions $|u(t,x)| \to 1$, where the evolution equation (\ref{nls-idd}) is singular. 

In the combined NLS model (\ref{nls-combined}) with sufficiently small $\epsilon > 0$ and $\delta \in (0,1)$, one can achieve both the local well-posedness and the spectral stability of the black soliton. Although the orbital stability problem might still be out of reach for $\epsilon > 0$ due to the mismatch between the function spaces for the energy and momentum conservation, the combined model (\ref{nls-combined}) can be used to perform study of the limiting transition $\delta \to 1$ to shed more light on the well-posedness theory for the NLS model (\ref{nls-idd}).

\vspace{0.25cm}

{\bf Acknowledgement.} The work of D. E. Pelinovsky is partially supported by AvHumboldt Foundation as Humboldt Reseach Award. The work of M. Plum is supported by Deutsche Forschungsgemeinschaft (German Research Foundation) - Project-ID 258734477 - SFB 1173.

\end{document}